\documentclass[12pt, reqno]{amsart}
\usepackage[utf8]{inputenc}
\usepackage[english]{babel}
 
\usepackage[square,sort,comma,numbers]{natbib}

\usepackage{soul}

\usepackage[dvipsnames]{xcolor}
\usepackage{float}

\usepackage{graphicx}
\usepackage{tikz}
\usetikzlibrary{arrows}
\usetikzlibrary{matrix,arrows,decorations.pathmorphing}

\usepackage{amsmath}
\usepackage{amssymb}
\usepackage{amsthm}

\usepackage{stmaryrd}
\usepackage[all]{xy}
\usepackage[mathcal]{eucal}
\usepackage{verbatim}\usepackage{fullpage} \usepackage{wrapfig}
\usepackage{lipsum}
\usepackage[all]{xy}
\theoremstyle{plain}
\newtheorem{thm}{Theorem}

\newtheorem{cor}[thm]{Corollary}

\theoremstyle{definition}

\theoremstyle{remark}

\newcommand{\nc}{\newcommand}
\nc{\dmo}{\DeclareMathOperator}

\DeclareMathOperator{\Mod}{Mod}

\DeclareMathOperator{\Homeo}{Homeo}

\nc{\para}[1]{\medskip\noindent\textbf{#1.}}
\title{Vanishing of the Euler class in Power subgroups of the punctured mapping class group}

\author{Lei Chen}
\address{\newline Department of Mathematics   \newline California Institute of Technology   \newline Pasadena, CA 91125,  USA }
\email{chenlei@caltech.edu}
\begin{document}
 \bibliographystyle{alpha}
 
\maketitle

\begin{abstract}
In this paper, we use a result of Dahmani to show that the Euler class of some power subgroup (the subgroup normally generated by a fixed power of Dehn twist about a non-separating curve) is trivial inside the mapping class group of once punctured surface. 
\end{abstract}
Let $S=S_{g,p}^b$ be a surface of genus $g$ with $p$ marked points and $b$ boundary components. We omit $p$ or $b$ whenever it is $0$. Let $\Mod(S_{g,p}^b)$ be the mapping class group of $S_{g,p}^b$; i.e., the group of connected components of homeomorphism groups of $S_{g,p}^b$ fixing $p$ marked points as a set and $b$ boundary components pointwise. 

Let  $P(S)(n)$ be the normal closure of $T_d^n$ in $\Mod(S)$ for a non-separating curve $d$ on $S$. For $\Mod(S_{g,1})$, there is an Euler class $E\in H^2(\Mod(S_{g,1});\mathbb{Z})$, which is the pullback of the Euler class in $H^2(\Homeo_+(S^1);\mathbb{Z})$ coming from the well-known action of $\Mod(S_{g,1})\to \Homeo_+(S^1)$ constructed by Nielsen. 
 
In this paper, we show the following.
\begin{thm}[The vanishing of the Euler class]\label{main}
For $g>1$, there exists $N_0$, such that for any multiple $N$ of $N_0$, we have that the Euler class vanishes in the $P(S_{g,1})(N)$.
\end{thm}
\begin{proof}
As discussed in \cite[Chapter 5.5]{FM}, the Euler class of
 $\Homeo_+(S^1)$ is the Euler class of the following central extension
\begin{equation}\label{1}
1\to \mathbb{Z}\to \widetilde{\Homeo}_+(S^1)\to \Homeo_+(S^1)\to 1,
\end{equation}
where $\widetilde{\Homeo}_+(S^1)$ consists of lifts of $\Homeo_+(S^1)$ through the universal cover $\mathbb{R}\to S^1$. 
The Euler class serves as an obstruction for the non-splitting  of \eqref{1}.

To show that the Euler class in $P(S_{g,1})(N)$ vanishes, we only need to show that the embedding $\phi: P(S_{g,1})(N)\to \Mod(S_{g,1})\to \Homeo_+(S^1)$ has a lift to $\widetilde{\Homeo}_+(S^1)$.

\vskip .3cm
We now construct the lift explicitly by Dahmani's result.

The rotation number of $f\in \widetilde{\Homeo}_+(S^1)$ is defined as the limit of $\frac{f^n(x)}{n}$ as $n\to \infty$ for any $x\in \mathbb{R}$. It does not depend on the choice of $x$ and for $g\in \Homeo_+(S^1)$, the rotation number is defined as the modulo $1$ of the rotation number of any lift of $g$ in $\widetilde{\Homeo}_+(S^1)$. We claim that the image $\phi(T_d^n)\in \Homeo_+(S^1)$ has zero rotation number mod $1$.

 The reason is that $\phi(T_d^n)$ is obtained as the restriction to the boundary of $\mathbb{H}^2$ of the lift of $T_d^n$ to the universal cover $\pi: \mathbb{H}^2\to S_{g}$ fixing a point. We know that the lift of $T_d^n$ to the universal cover is identity on one component $C$ of the preimage of $S_{g}-d$ under $\pi$. This means that $\phi(T_d^n)$ is identity on the limit points of $C$, which is a cantor set. Since $\phi(T_d^n)$ has a fixed point, we know that the rotation number of $\phi(T_d^n)$ is zero mod $1$.

We choose the image of $\tilde{\phi}(T_d^n)$ to be the unique lift in  $\widetilde{\Homeo}_+(S^1)$ such that the rotation number is zero. We claim that such map $\tilde{\phi}$ induces a homomorphism
 $P(S_{g,1})(N)\to \widetilde{\Homeo}_+(S^1)$ as a lift of $\phi$. We only need to check that relations are preserved. 
 
W will use the result of Dahmani \cite{Dahmani} about deep relations. He proved that there exists $N_0$ such that for any multiple $N$ of $N_0$, the subgroup $P(S_{g,1})(N)$ of $\Mod(S_{g,1})$ has a presentation with the generating set 
\[
\{T_d^n| \text{ $d$ non separating curve}\}, \] 
and relations only ``commutating relations" and ``conjugation relations".

There are only two kinds of relations by Dahmani's result, where we check each one.
\begin{itemize}
\item Commutating relations: $T_d^n,T_c^n$ commute if and only if $c,d$ disjoint. Rotation numbers are additive for commuting elements. Therefore commutating relations hold under $\tilde{\phi}$ because both $\tilde{\phi}(T_d^nT_c^n)$ and $\tilde{\phi}(T_c^nT_d^n)$ have zero rotation numbers.
 \item Conjugation relations: $T_d^nT_c^nT_d^{-n}=T_{T_d(c)}^n$. Rotation number is a conjugacy invariant.
Therefore conjugation relations hold under $\tilde{\phi}$ because both $\tilde{\phi}(T_d^nT_c^nT_d^{-n})$ and $\tilde{\phi}(T_{T_d(c)}^n)$ have zero rotation numbers. \end{itemize}

The construction shows that $\phi$ has a lift $\tilde{\phi}$, which implies that the Euler class is trivial on $P(S_{g,1})(N)$.
\end{proof}
\begin{cor}
For $g>1$, there exists $N_0$, such that for any multiple $N$ of $N_0$, we have that in $\Mod(S_g^1)$, no power of the Dehn twist of the boundary curve lies in $P(S_{g}^1)(N)$.
\end{cor}
\begin{proof}
As discussed in \cite[Chapter 5.5.6]{FM}, the Euler class of $\Mod(S_{g,1})$ is the associated Euler class of the following central extension
\[
1\to \mathbb{Z}\to \Mod(S_{g}^1)\xrightarrow{F} \Mod(S_{g,1})\to 1.
\]
The image of  $P(S_{g}^1)(N)$ under $F$ is $P(S_{g,1})(N)$ with the same generators. The lift we construct in the proof of Theorem \ref{main} is an inverse of $F:P(S_{g}^1)(N)\to P(S_{g,1})(N)$, which is also a surjection. Therefore, we know that $F$ is an isomorphism, which proves this corollary.
\end{proof}
Theorem \ref{main} also implies a special case of \cite[Corollary 1.2]{Funar}.
\begin{cor}
For $g>1$, there exists $N_0$, such that for any multiple $N$ of $N_0$, we have that $P(S_{g,1})(N)<\Mod(S_{g,1})$ has infinite index.
\end{cor}
\begin{proof}
The Euler class is nontrivial rationally in $H^2(\Mod(S_{g,1});\mathbb{Q})$, which means that it is nontrivial in any finite index subgroup of $\Mod(S_{g,1})$. Then Theorem \ref{main} implies this corollary.
\end{proof}

\para{Acknowledgment} The author would like  to thank Fran\c cois Dahmani and Vlad Markovic for very useful discussions.

\bibliography{citing}{}

\end{document}